\documentclass[a4paper,11pt]{amsart}

\usepackage{amsmath}
\usepackage{amsfonts}
\usepackage{amssymb}
\usepackage{graphicx}
\usepackage[abbrev,alphabetic]{amsrefs}%refarence.
\usepackage{color}%erase
\usepackage{soul,xcolor} %erase
\setstcolor{blue}

\usepackage{amsthm}
\usepackage{comment}
\usepackage[all,cmtip]{xy}
\usepackage{tikz-cd}
\usetikzlibrary{cd}

\newcommand{\mbN}{\mathbb{N}}
\newcommand{\mbQ}{\mathbb{Q}}

\newcommand{\mbZ}{\mathbb{Z}}

\newcommand{\mcE}{\mathcal{E}}
\newcommand{\mcF}{\mathcal{F}}

\newcommand{\mcI}{\mathcal{I}}
\newcommand{\mcL}{\mathcal{L}}
\newcommand{\mcO}{\mathcal{O}}

\DeclareMathOperator{\Supp}{Supp}
\DeclareMathOperator{\Spec}{Spec}

\DeclareMathOperator{\GL}{GL}

\DeclareMathOperator{\Exc}{Exc}

\newcommand*{\sheafhom}{\mathcal{H}\kern -.5pt om}
\newcommand*{\coloneq}{\mathrel{\mathop:}=}

\theoremstyle{plain}
\newtheorem{theorem}{Theorem}[section]
\newtheorem{proposition}[theorem]{Proposition}
\newtheorem{lemma}[theorem]{Lemma}
\newtheorem{corollary}[theorem]{Corollary}

\newtheorem{question}[theorem]{Question}

\theoremstyle{definition}
\newtheorem{definition}[theorem]{Definition}

\theoremstyle{remark}
\newtheorem{remark}[theorem]{Remark}

\title[On the rationality of Kawamata log terminal singularities]{On the rationality of Kawamata log terminal singularities in positive characteristic}
\author{Christopher Hacon}
\address{Department of Mathematics \\
University of Utah\\
Salt Lake City, UT 84112, USA}
\email{hacon@math.utah.edu}\email{}

\author{Jakub Witaszek}
\address{Department of Mathematics, Imperial College, London, 180 Queen's Gate, 
London SW7 2AZ, UK} 
\email{j.witaszek14@imperial.ac.uk}

\begin{document}

\begin{abstract}
We show that there exists a natural number $p_0$ such that any three-dimensional Kawamata log terminal singularity defined over an algebraically closed field of characteristic $p>p_0$ is rational and in particular Cohen-Macaulay.
\end{abstract}


\maketitle

\section{Introduction}
One of the main goals of algebraic geometry is to understand the structure of smooth projective varieties. The techniques of the 
Minimal Model Program, MMP for short, play a fundamental role in the pursuit of this objective. Unluckily,
starting from dimension three, to apply the techniques of the MMP it is necessary to consider varieties with mild singularities.
From the point of view of the MMP, the most natural class of singularities is that of Kawamata log terminal (klt) singularities, which includes canonical and terminal singularities. 
In characteristic zero, a large part of the MMP for varieties with klt singularities is known to hold (see \cite{bchm06}). In this context it is of course essential to have a good understanding of the properties of klt singularities. Perhaps the most important of these properties is the fact that klt singularities are rational and, in particular, Cohen-Macaulay. This fundamental fact was first proved by Elkik for canonical singularities and was later generalized to klt singularities,  see for example \cite[Theorem 5.22]{KM98} and references therein. 

The main technical tool used in the proof of the results of the MMP (and in the proof of Elkik's result) is the famous Kawamata-Viehweg vanishing which is known to fail in positive characteristic and dimension $\geq 2$.
So it is not surprising that the situation in positive characteristic $p>0$  is much more complicated. In spite of this, it is known that the klt MMP is valid in dimension two (\cite{tanaka12}) and dimension three when the characteristic $p > 5$ (\cite{hx13}, \cite{ctx13}, \cite{birkar13}, \cite{BW17}). It is natural to wonder if in positive characteristic, klt singularities  are rational. In dimension two, this is known to hold for all characteristics, however  
recently, it was shown that this does not hold when $p= 2$ and the dimension is at least three \cite{GNT06}, \cite{kovacs17}, and \cite{CT06PLT} (see Remark \ref{remark:quotient}). 
On the positive side, it is known that klt threefold singularities are Witt-rational when $p>5$ (\cite{GNT06}). This interesting result allows for the extension of Esnault's results on counting points over finite fields to singular varieties. Note that rational singularities are Witt-rational but the converse is not true.
Furthermore, it is known that strongly F-regular singularities are rational.
Note that F-regular singularities are klt but klt threefold singularities need not be strongly F-regular even in large characteristic (\cite{CTW15a}).

The goal of this paper is to show the following result.
\begin{theorem} \label{theorem:main} There exists a natural number $p_0>0$ such that for any three-dimensional Kawamata log terminal pair $(X,\Delta)$ defined over an algebraically closed field of characteristic $p>p_0$, we have that $X$ has rational and, in particular, Cohen-Macaulay singularities. Moreover, if $X$ is $\mbQ$-factorial, then, for any divisor $D$, the sheaf $\mathcal{O}_X(D)$ is Cohen-Macaulay.
\end{theorem}
It is not known whether the above theorem holds or is false if we just assume that $p>2$. The main tool in the proof of this result is the Kawamata-Viehweg vanishing for log del Pezzo surfaces in large characteristic. The constant $p_0$ comes from this result.
 \begin{theorem}[{\cite{CTW15b}*{Theorem 1.2}}] \label{theorem:CTW} There exists a constant $p_0>0$ with the following property.

 Let $X$ be a projective surface of Fano type defined over an algebraically closed field of characteristic $p>p_0$. Let $\Delta$ be an effective $\mbQ$-divisor such that $(X,\Delta)$ is klt, and let $L$ be a Weil divisor for which $L-(K_X+\Delta)$ is nef and big. Then $H^i(X,\mcO_X(L))=0$ for $i>0$.
\end{theorem}
We say that a surface $X$ is of Fano type, if there exists a $\mbQ$-divisor $B$ such that $(X,B)$ is klt and $-(K_X+B)$ is ample. Note that the Kawamata-Viehweg vanishing does not hold for rational surfaces in general even in large characteristic (\cite{CT07KV}).

Let us also note that it has been recently shown in \cite{kovacs17b}, that Cohen-Macaulay klt singularities are rational. The results from this paper allowed us to simplify the proof of Theorem \ref{theorem:main}.\\

One possible application of these results stems from Minimal Model Program in  mixed characteristic. Given a positive-characteristic variety $X$, which lifts to characteristic zero, and a basic operation of the MMP $f \colon X \dashrightarrow Y$ such as a contraction or a flip, it is natural to ask if $f$ lifts to characteristic zero as well. As a corollary of Theorem \ref{theorem:main}, we show that this holds for divisorial and flipping contractions of threefolds in large characteristic (see Corollary \ref{c:lift}). Unfortunately, we do not know how to deal with the case of flips.

The article is organised as follows. In Section 2, we discuss rational and Cohen-Macaulay singularities, strong F-regularity, plt blow-ups, and liftings to the rings of Witt vectors. In Section 3, we prove the existence of certain short exact sequences on plt threefolds (see Subsection \ref{subsection:sketch}). In Section 4, we finish the proof of Theorem \ref{theorem:main}. In Section 5 we explore applications to the MMP in mixed characteristic.

\subsection{Sketch of the proof} \label{subsection:sketch}
Assume that $X$ is $\mbQ$-factorial. Consider a \emph{plt blow-up} (Proposition \ref{proposition:plt-blow-up}) at a closed point $x \in X$, that is a birational morphism $f \colon Y \to X$ with an irreducible exceptional divisor $S$ such that
\begin{itemize}
	\item $(Y,S)$ is plt and $\mbQ$-factorial,
	\item $-(K_Y+S)$ is $f$-ample,
	\item $f$ is an isomorphism over $X\backslash \{x\}$.
\end{itemize}
By adjunction, there exists a $\mbQ$-divisor $\Delta_{\mathrm{diff}}$ such that $(S, \Delta_{\mathrm{diff}})$ is a klt log del Pezzo pair.

When $S$ is Cartier, one can apply the following strategy (cf.\ \cite{GNT06} and \cite{CT06PLT}). For every $k\geq 0$, consider
\[
0 \to \mcO_Y(-(k+1)S) \to \mcO_Y(-kS) \to \mcO_S(G_k) \to 0,
\]
where $G_k \sim -kS|_S$. Since
\[
G_k\sim _{\mathbb Q}K_S + \Delta_{\mathrm{diff}} \underbrace{- (K_S + \Delta_{\mathrm{diff}}) + G_k}_{\text{ample}},
\]
by Theorem \ref{theorem:CTW}, we have that $H^1(S, \mcO_S(G_k)) =0$
for every $k\geq 0$, and so $R^1f_*\mcO_Y(-(k+1)S) \to R^1f_*\mcO_Y(-kS)$ is surjective. Since $-S$ is $f$-ample, $R^1f_*\mcO_Y(-kS)=0$ for $k\gg 0$ by Serre vanishing, which implies $R^1f_* \mcO_Y=0$ as well. 
By \cite{hx13} one sees that $Y$ has F-regular singularities along $S$ and hence $Y$ has rational singularities along $S$. Therefore, if $X$ is Cohen-Macaulay, then it has rational singularities at $x$. The fact that  $X$ is Cohen-Macaulay at $x$ follows by a similar argument.

When $S$ is not Cartier, the above short exact sequence may fail to be exact, however we show the existence of the following short exact sequence (Proposition \ref{proposition:3folds-ses})
\begin{equation} \label{eq:ses}
0 \to \mcO_Y(-(k+1)S) \to \mcO_Y(-kS) \to \mcO_S(G_k) \to 0,
\end{equation}
where $G_k$ is a Weil divisor satisfying $G_k \sim_{\mbQ} -kS|_S - \Delta_k$ for some effective $\mbQ$-divisor $\Delta_k$ on $S$ such that $\Delta_k \subseteq \Delta_{\mathrm{diff}}$. Therefore, we can apply the same argument as above with minor changes. 

When $X$ is not $\mbQ$-factorial, we apply Proposition \ref{proposition:plt-blow-up}, Proposition \ref{proposition:non-Q-factorial-case}, and the relative F-inversion of adjunction (Lemma \ref{lemma:relative-inversion-of-adjunction}).  

 \section{Preliminaries}
We say that a scheme $X$ is a \emph{variety} if it is integral, separated, and of finite type over a field $k$. Throughout this paper, $k$ is an arbitrary field of positive characteristic $p>0$. For a $k$-variety $X$, we denote the relative canonical divisor $K_{X/\Spec k}$ by $K_X$ (see \cite{kollar13}*{Definition 1.6}).

We refer to \cite{KM98} for basic definitions in birational geometry, and to \cite{kollar13} whenever $k$ is not algebraically closed. We say that $(X,\Delta)$ is a \emph{log pair} if $X$ is normal, $\Delta$ is an effective $\mbQ$-divisor, and $K_X+\Delta$ is $\mathbb{Q}$-Cartier.

Given a variety $X$ and a Weil divisor $D$ on $X$, we say that $f \colon Y \to X$ is a \emph{log resolution} of $(X,D)$ if $Y$ is regular, $\mathrm{Ex}(f)$ has pure codimension one and $(Y, \Supp(f^{-1}(D)+\mathrm{Exc}(f)))$ has simple normal crossings (see \cite{kollar13}*{Definition 1.8 and 1.12}). If $X$ is of dimension at most three, then a log resolution of $(X,D)$ always exists (see \cite{lip78}, \cite{cutkosky09}, \cite{CP08}, and \cite{CP09}).

\begin{theorem}[{\cite[Theorem 10.4]{kollar13}}] \label{thm:relKV} Let $X$ be a regular surface over a field $k$ and let $f \colon  X \to Y$ be a proper, generically finite morphism with exceptional curves $C_i$ such that $\bigcup_i C_i$ is connected. Let $L$ be a line bundle on $X$ and assume that there exist $\mbQ$-divisors $N$ and $\Delta = \sum d_i C_i$ such that:
\begin{itemize}
	\item $L \equiv_{f} K_X + \Delta + N$,
	\item $N \cdot C_i \geq 0$ for every $i$,
	\item $ 0 \leq d_i < 1$ for every $i$. 
\end{itemize}
Then $R^1f_*L = 0$.
\end{theorem}

\subsection{Rational and Cohen-Macaulay singularities} \label{ss:rat-cm-sing}
We say that a variety $X$ has \emph{rational singularities} if it is Cohen-Macaulay and there exists a log resolution $f \colon Y \to X$ such that $Rf_*\mcO_Y = \mcO_X$. %We say that $X$ satisfies \emph{GR-vanishing} if there exists a log resolution $f \colon Y \to X$ such that $R^if_*\omega_Y = 0$ for $i>0$. 
%In the above definitions, 
Assuming the existence of resolutions of singularities, the phrase \emph{there exists} can be replaced by \emph{for every} (this follows from the main result in \cite{CR15}). For the definition of Serre's condition $S_d$ and Cohen-Macaulay singularities we refer to \cite[Definition 5.1 and Definition 5.2]{kollar13}.

\begin{proposition} \label{proposition:CM_kollar}
Let $f \colon Y \to X$ be a proper birational morphism of normal varieties and let $D$ be a Weil divisor on $Y$ such that $\mcO_Y(D)$ is Cohen-Macaulay. Assume that $R^if_*\mcO_Y(D)=0$ and $R^if_*\mcO_Y(K_Y-D)=0$ for all $i>0$. Then $f_*\mcO_Y(D)$ is Cohen-Macaulay.
\end{proposition}
\begin{proof} Since $D$ is Cohen-Macaulay,  so is $\sheafhom_{\mathcal{O}_Y}(\mcO_Y(D), \omega_Y)$ (see \cite[Corollary 5.70]{KM98}). In particular, the latter sheaf is reflexive, and hence it is isomorphic to $\mcO_Y(K_Y-D)$. Therefore, the proposition follows from \cite[Theorem 2.74]{kollar13}.
\end{proof}

The following result is well known.
\begin{proposition}[{\cite[Lemma 8.1]{kovacs17b}}] \label{prop:rational_and_CM_is_good} Let $X$ be a variety with rational singularities. Then $Rf_*\omega_Y = \omega_X$ for every resolution $f \colon Y \to X$.
\end{proposition} 

\begin{comment}
Let $\omega_X^{\bullet}$ and $\omega_Y^{\bullet}$ be the dualizing complexes. The assumption that $X$ is Cohen-Macaulay implies that $\omega_X^{\bullet} \simeq \omega_X[d]$ for $d = \dim X$.

By Grothendieck duality,
\[
Rf_* \omega_Y[d] \simeq Rf_* R\sheafhom_Y(\mcO_Y, \omega_Y^{\bullet}) \simeq R\sheafhom_X(Rf_*\mcO_Y, \omega_X^{\bullet}) \simeq \omega_X^{\bullet}\cong  \omega_X[d].
\]
\end{comment}

Let us recall that a rank one sheaf  $\mathcal{F}$ on a normal variety $X$ is reflexive if and only if it is divisorial, that is of the form $\mcO_X(D)$ for some Weil divisor $D$.  Moreover, $i_*(\mathcal{F}|_{U}) \simeq \mcF$, where $i \colon U \to X$ is an inclusion of an open subset whose complement is of codimension at least two.

\begin{lemma} \label{lemma:CM} Let $X$ be a normal variety, let $S \subset X$ be a prime divisor, and let $D$ be any Weil divisor. For the following exact sequence
\[
0 \to \mcO_X(-S + D) \to \mcO_X(D) \to \mcE \to 0,
\]
we have $\Supp \mcE = S$. Moreover, if $\mcO_X(-S+D)$ and $\mcO_X(D)$ satisfy Serre's condition $S_3$, then $\mcE$ is reflexive as a sheaf on $S$.
\end{lemma}
\begin{proof}
Since the ideal sheaf $\mcI_S$ of $S$ annihilates $\mcE$ and the sequence above is the standard restriction exact sequence outside of a closed subset of codimension two, we have that $\Supp(\mcE)=S$. The second part of the lemma follows from \cite{kollar13}*{Lemma 2.60}.
\end{proof}

\subsection{Strongly F-regular singularities}
The proof of the main theorem makes use of the concept of F-regular singularities.  We refer the reader to \cite{schwedetucker12} for a comprehensive treatment of this topic. The sole purpose of introducing them here is to show, summing up known results, that a plt threefold singularity $(X,S)$ in characteristic $p>5$ is rational near $S$ and, assuming $\mbQ$-factoriality, any reflexive rank one sheaf on $X$ is Cohen-Macaulay near $S$ (Proposition \ref{proposition:f-regular_are_rational}, Proposition \ref{proposition:plt_are_f-regular}, and, for the non-$\mbQ$-factorial case, Proposition \ref{proposition:relative-rationality}). This shows that singularities of plt blow-ups possess these desirable properties (see Proposition \ref{proposition:plt-blow-up} and Proposition \ref{proposition:non-Q-factorial-case}). 
For the convenience of the reader, we recall the basic definitions.
\begin{definition} For an F-finite scheme X of characteristic $p>0$ and an effective $\mbQ$-divisor $\Delta$, we say that $(X,\Delta)$ is \emph{globally F-split} if for every $e \in \mbZ_{>0}$, the natural morphism
\[
\mcO_X \to F^e_* \mcO_X(\lfloor(p^e-1)\Delta \rfloor)
\]
splits in the category of sheaves of $\mcO_X$-modules.

We say that $(X,\Delta)$ is \emph{globally F-regular} if for every effective divisor $D$ on $X$ and every big enough $e \in \mbZ_{>0}$, the natural morphism
\[
\mcO_X \to F^e_* \mcO_X(\lfloor(p^e-1)\Delta \rfloor + D)
\]
splits.

We say that $(X,\Delta)$ is \emph{purely globally F-regular}, if the definition above holds for those $D$ which intersect $\lfloor \Delta \rfloor$ properly (see \cite{Das15}*{Definition 2.3(2)}).
\end{definition}
The local versions of the above notions are called F-purity, strong F-regularity, and pure F-regularity, respectively. Moreover, given a morphism $f \colon X \to Y$, we say that $(X,\Delta)$ is F-split, F-regular, or purely F-regular over $Y$, if the corresponding splittings hold locally over $Y$ (see \cite{hx13}*{Definition 2.6}).

\begin{remark} \label{remark:relative_f-split} Assume that $(X,\Delta)$ is F-split over $Y$ with respect to $f \colon X \to Y$. Then for any affine open subset $U \subseteq Y$, we have that $(f^{-1}(U), \Delta|_{f^{-1}(U)})$ is globally F-split (see \cite{hx13}*{Proposition 2.10}, cf.\ \cite{cgs14}*{Remark 2.8}). In particular, if $u \colon Y \to Y'$ is affine, then $(X,\Delta)$ is also F-split over $Y'$ with respect to $u \circ f$. Analogous statements hold for F-regularity and pure F-regularity.
\end{remark}

\begin{proposition} \label{proposition:f-regular_are_rational} Let $X$ be a strongly F-regular variety. Then $X$ has rational singularities. Moreover, if $X$ is $\mbQ$-factorial and of dimension at least three, then any reflexive sheaf of rank one on it is $S_3$.
\end{proposition}
In particular, a reflexive sheaf of rank one on a $\mbQ$-factorial strongly F-regular threefold is Cohen-Macaulay.
\begin{proof}
$X$ is F-rational (see \cite[Appendix C]{schwedetucker12}), and hence it has rational singularities by \cite{kovacs17b}*{Corollary 1.12}.

The last statement in the proposition is a consequence of \cite[Theorem 3.8]{patakfalvi-schwede14}. Indeed, let $\mathcal{O}_X(D)$ be a reflexive sheaf, and take a general effective divisor $E \sim -(mp+1)D$ for $m \gg 0$. Then $(X,\frac{1}{mp+1}E)$ is strongly F-regular, and the result follows.
\end{proof}

\begin{proposition} \label{proposition:plt_are_f-regular} Let $(X,S)$ be a plt three-dimensional pair defined over an algebraically closed field of characteristic $p>5$. Then $S$ is normal and $X$ is strongly F-regular along it.
\end{proposition}
\begin{proof}
By \cite{hara98} and \cite[Theorem A]{Das15} (cf.\ Lemma \ref{lemma:relative-inversion-of-adjunction}), we have that $(X,S)$ is purely $F$-regular along $S$ and $S$ is normal. This implies F-regularity of $X$ along $S$.
\end{proof}

In order to deal with non-$\mbQ$-factorial singularities, we also need a relative version of the above result.
\begin{proposition} \label{proposition:relative-rationality} Let $(X,S)$ be a plt three-dimensional pair defined over an algebraically closed field of characteristic $p>5$, and let $f \colon X \to Z$ be a proper birational morphism between normal varieties such that $\dim f(S)=2$. If $-(K_X+S)$ is $f$-ample, then $Z$ is strongly F-regular along $f(S)$.
\end{proposition}
This does not follow from Proposition \ref{proposition:plt_are_f-regular} directly, because $K_Z + f(S)$ is almost never $\mbQ$-Cartier. Further, let us note that we will need to apply Proposition 
\ref{proposition:relative-rationality} in the case when $X$ is not $\mbQ$-factorial.
\begin{proof}
As above, $S$ is normal. The normalization morphism $f(S)^\nu \to f(S)$ is affine, and so, by \cite{hx13}*{Theorem 3.1}, the pair $(S, \Delta_{\mathrm{diff}})$, defined by adjunction $K_S + \Delta_{\mathrm{diff}} = (K_X + S)|_S$, is globally $F$-regular over $f(S)$ (see Remark \ref{remark:relative_f-split}). Lemma \ref{lemma:relative-inversion-of-adjunction} and \cite{hx13}*{Lemma 2.12} imply that $Z$ is strongly $F$-regular along $f(S)$.
\end{proof}

In the proposition above, we used a relative version of inversion of adjunction, which is a consequence of \cite{Das15}*{Theorem B}. The proof is very similar to \cite{Das15}*{Corollary 5.4} and \cite{CTW15b}*{Lemma 2.7}, but we include it here for the convenience of the reader.
\begin{lemma} \label{lemma:relative-inversion-of-adjunction}
Let $(X,S+B)$ be a plt pair where $S$ is a prime divisor, and let $f \colon X \to Z$ be a proper birational morphism between normal varieties. Assume that $-(K_X+S+B)$ is $f$-ample and $(\overline{S},B_{\overline{S}})$ is globally $F$-regular over $f(S)$, where $\overline{S}$ is the normalization of $S$, and $B_{\overline{S}}$ is defined by adjunction $K_{\overline{S}} + B_{\overline{S}} = (K_X + S + B)|_{\overline{S}}$. Then $(X,S+B)$ is purely globally F-regular over  a Zariski-open neighbourhood of $f(S) \subseteq Z$. \end{lemma}

\begin{proof}
Since the question is local over $Z$, we can assume that $Z$ is affine, that is $Z = \Spec(R)$ for some ring $R$. Further, by perturbing the coefficients of $B$, we can assume that the Cartier index of $K_X + S + B$ is not divisible by $p$ (see for example the proof of \cite[3.8]{Das15}). By \cite[Theorem A]{Das15}, $(X,S+B)$ is purely F-regular, and $S$ is a normal F-pure centre. By  standard arguments (replacing $B$ by $B+\epsilon H$ for $0 < \epsilon \ll 1$ and a very ample divisor $H$ intersecting $S$ properly), it is enough to show that $(X,S+B)$ is $F$-split over a Zariski-open neighbourhood of $f(S)$ (see the proof of \cite{schwedesmith10}*{Theorem 3.9}). 

Set $\mcL \coloneq \mcO_X((1-p^e)(K_X+S+B))$ and for a sufficiently divisible integer $e\gg 0$ consider the following diagram 

{
\begin{center}
\begin{tikzcd}
H^0(X,F_*^e\mcL) \arrow{r} \arrow{d}{\psi^X_{S+B}} & H^0(S,F_*^e(
\mcL|_S)) \arrow{d}{\psi^S_{B_S}} \arrow{r} & H^1(X,F_*^e(\mcL(-S))) \\
 H^0(X,\mcO_X)  \arrow{r} & H^0(S,\mcO_S), &
\end{tikzcd}
\end{center}}
\noindent where $\psi^X_{S+B}$ and $\psi^S_{B_S}$ are the trace maps, while the horizontal arrows come from the restriction exact sequences. The diagram is well defined and commutes by the fact that $S$ is an F-pure centre and by the equality of the Different and the F-different (see \cite{Das15}*{Theorem B}, cf.\ \cite{CTW15b}*{Subsection 2.2}). Note that $\psi^S_{B_S}$ is surjective, since $(S, B_S)$ is F-split over $f(S)$.

Since $e \gg 0$, the relative Serre vanishing implies that $H^1(X,F_*^e(\mcL(-S))) = 0$. In particular, the upper left horizontal arrow is surjective, and so the composition
\[
H^0(X, F^e_* \mcL) \xrightarrow{\psi^X_{S+B}} H^0(X,\mcO_X) \to H^0(S,\mcO_S)
\] is surjective as well. But $H^0(X,\mcO_X)= H^0(Z,\mcO_Z)=R$, hence this implies that the zero locus of the ideal in $R$ generated by $\mathrm{im}(\psi^X_{S+B})$ is disjoint from $f(S)$. After removing this locus, the surjectivity of $\psi^X_{S+B}$ follows.
\end{proof}

\subsection{Local to global transition}
In the course of the proof of the main theorem we will need the following two results that allow for the transition from a local to a global case.

\begin{proposition}[Generalized plt blow-up] \label{proposition:plt-blow-up}
Let $(X,\Delta)$ be a Kawamata log terminal three-dimensional pair defined over an algebraically closed field of characteristic $p>5$, and let $x \in X$ be a closed point. Assume that $X \backslash \{x\}$ is $\mbQ$-factorial.

Then there exists a projective birational morphism $f \colon Y \to X$ and an effective $\mbQ$-divisor $\Delta_Y$ on $Y$ such that
\begin{itemize}
	\item $f$ is an isomorphism over $X \backslash \{x\}$,
	\item $S \coloneq \Exc f$ is irreducible and anti-$f$-nef,
	\item $(Y, S+\Delta_Y)$ is a $\mbQ$-factorial plt threefold,
	\item $-(K_Y+S+\Delta_Y)$ is $f$-ample.
\end{itemize}
Furthermore, $S$ is normal and $Y$ is strongly F-regular along it.
\end{proposition}
\noindent When $x \in X$ is $\mbQ$-factorial, then $-S$ is automatically $f$-ample.
\begin{proof} Apart from the last statement, this is a direct consequence of \cite{GNT06}*{Proposition 2.15}.  The strong F-regularity of $Y$ and the normality of $S$ follow from Proposition \ref{proposition:plt_are_f-regular}.
\end{proof}

For the non-$\mbQ$-factorial case of the main theorem, we will also need the following result.
\begin{proposition} \label{proposition:non-Q-factorial-case} Under the assumptions of Proposition \ref{proposition:plt-blow-up}, the $f$-relative semiample fibration $g \colon Y \to Y'$ associated to $-S$ exists. Moreover, $g$ is small, $-S|_S$ is big, and $Y'$ is strongly F-regular along $g(S)$. 
\end{proposition}
\begin{proof}
The first statement follows  from the base point free theorem (see \cite[Theorem 2.9]{GNT06}) after having noticed that $-S = K_Y + \Delta_Y -(K_Y+S+\Delta_Y)$.

For the proof of the bigness of $-S|_S$ it is enough to show that $g$ is small. To this end, run a $(K_Y+S+\Delta_Y)$-MMP over $Y'$, and let $Y \dashrightarrow Y_{\mathrm{min}} \xrightarrow{g'} Y'$ be the minimal model. Since $-(K_Y+S+\Delta_Y)$ is $g$-ample, $g'$ is small. Moreover, as $S$ is $g$-numerically trivial, none of the steps of the MMP may contract $S$. Indeed, an MMP preserves $\mbQ$-factoriality, and hence a contracted divisor is always relatively anti-ample with respect to a divisorial contraction. Since $\Exc f = S$ is irreducible, this shows that $g$ is small and it does not contract $S$.

For the last statement, we replace $Y$ by the output of a $-(K_Y+S)$-MMP over $Y'$, and so we can assume that $-(K_Y+S)$ is $g$-nef and $g$-big. We can run such an MMP, since we can write
\[
-\epsilon (K_Y+S)=K_Y+S+(1+\epsilon)\Delta_Y-(1+\epsilon)(K_Y+S+\Delta_Y),
\]
 where $-(K_Y + S + \Delta_Y)$ is $g$-ample and $(Y,S +(1+\epsilon)\Delta_Y)$ is plt for $0 < \epsilon \ll 1$.  Further, by replacing $Y$ by the image of the $g$-relative semiample fibration associated to $-(K_Y+S)$, we can assume that $-(K_Y+S)$ is $g$-ample. The fibration exists by the klt base point free theorem, as $S$ is $g$-numerically trivial and we can write
\[
-(K_Y+S) = K_Y - 2(K_Y+S) + S.
\]
Although $Y$ may have ceased to be $\mbQ$-factorial, $(Y,S)$ is a well defined plt pair (in particular, $K_Y+S$ is $\mbQ$-Cartier), and so Proposition \ref{proposition:relative-rationality} concludes the proof.
\end{proof}

\subsection{Liftings}
Let $W_m(k)$ denote the ring of Witt vectors of length $m$. We say that a $k$-variety $X$ \emph{lifts over $W_m(k)$} if there exits a flat morphism $\widetilde{X} \to \Spec W_m(k)$ such that the special fibre is isomorphic to $X$. Similarly, we can define liftings of morphisms.

In Section \ref{section:applications} we will need the following result.
\begin{proposition}[{\cite[Proposition 2.1]{LS14}, \cite[Theorem 3.1]{CS09}}] \label{proposition:pushing_lifts} Let $f \colon Y \to X$ be a morphism of schemes satisfying $Rf_* \mcO_Y = \mcO_X$. Assume that $Y$ lifts to a scheme $\widetilde{Y}_m$ over $W_m(k)$ for some natural number $m \in \mbZ_{>0}$. Then $f$ lifts to  $\widetilde{f}_m \colon \widetilde{Y}_m \to \widetilde{X}_m$ over $W_m(k)$.
\end{proposition}
The same is true for formal lifts to the total Witt ring $W(k)$.
\begin{comment}
\begin{proof}
We pick $p_0$ as in Theorem \ref{theorem:main} and proceed by induction. We assume that $\pi$ lifts to ${\widetilde{\pi}}_{l-1} \colon \widetilde{X}_{l-1} \to \widetilde{Y}_{l-1}$ over $W_{l-1}(k)$ for $l \leq m$, and we would like to show that it lifts to $W_l(k)$. 

To this end, we consider the following exact sequence of sheaves of algebras on $\widetilde{Y}_{m-1}$:
\[
0 \to \pi_* \mcO_X \to ({\widetilde{\pi}}_{l-1})_*\mcO_{\widetilde{X}_l} \to ({\widetilde{\pi}}_{l-1})_*\mcO_{\widetilde{X}_{l-1}} \to R^1\pi_* \mcO_X,
\]
where $\widetilde{X}_l$ is the lift of $X$ to $W_l(k)$ induced by $\widetilde{X}_m$.

Since $X$ and $Y$ have rational singularities (Theorem \ref{theorem:main}), the Leray spectral sequence shows that $R^1\pi_* \mcO_X = 0$. Thus, we can take 
\[
\widetilde{Y}_l \coloneq \Spec\Big(({\widetilde{\pi}}_{l-1})_*\mcO_{\widetilde{X}_l}\Big).
\]
\end{proof}
\end{comment}

\section{The restriction short exact sequence} \label{section:surfaces}
\begin{proposition} \label{proposition:3folds-ses} Let $(X,S)$ be a plt three-dimensional $\mbQ$-factorial pair defined over an algebraically closed field $k$ of characteristic $p>5$, where $S$ is a prime divisor, and let $\Delta_{\mathrm{diff}}$ be the different. Then for any Weil divisor $D$ on $X$, there exists an effective $\mbQ$-divisor $\Delta_S \leq \Delta_{\mathrm{diff}}$ and a Weil divisor $G$ on $S$ such that $G\sim _{\mathbb Q} K_S+\Delta _S+D|_S$ and the sequence
\[
 0 \to \mcO_X(K_X+D) \to \mcO_X(K_X+S+D) \to \mcO_S(G) \to 0
\]
is exact.
\end{proposition}
Let us note that $G$ is only defined up to linear equivalence. 
\begin{proof}
Note that in proving the proposition, we are free to replace $X$ by an appropriate neighborhood of $S$.
By Proposition \ref{proposition:plt_are_f-regular}, we may assume that $X$ is strongly F-regular and $S$ is normal. Replacing $D$ by a linearly equivalent divisor, we may assume that $S$ is not contained in the support of $D$ and so $D|_S$ is a well defined $\mathbb Q$-divisor. Consider the following short exact sequence (see Lemma \ref{lemma:CM})
\[
0 \to \mcO_X(K_X + D) \to \mcO_X(K_X+S+D) \to i_*\mcE \to 0,
\]
where $\mcE$ is a sheaf supported on $S$, and $i \colon S \to X$ is the inclusion. By Proposition \ref{proposition:f-regular_are_rational} any reflexive rank one sheaf on $X$ satisfies Serre's condition $S_3$  and so,  by Lemma \ref{lemma:CM}, $\mcE$ is reflexive on $S$.

Let $\pi \colon \overline{X} \to X$ be a log resolution of $(X,S)$ and let $\overline{S}$ be the strict transform of $S$. Define $\Delta_{\overline{X}}$ by 
\[
K_{\overline{X}} + \overline{S} + \Delta_{\overline{X}} = \pi^*(K_X+S).
\] 
Consider the following exact sequence
\[
0 \to \mcO_{\overline{X}}(K_{\overline{X}} + \lceil \pi^*D \rceil ) \to \mcO_{\overline{X}}(K_{\overline{X}} + \overline{S} + \lceil \pi^*D \rceil) \to \mcO_{\overline{S}}(K_{\overline{S}}  + \lceil \pi^*D \rceil|_{\overline{S}}) \to 0.
\]
%where $j \colon \overline{S} \to \overline{X}$ is an inclusion and $\mcE_{\overline{S}} \coloneq \mcO_{\overline{S}}(K_{\overline{S}}  + \lceil \pi^*D \rceil|_{\overline{S}})$.

Since $(X,S)$ is plt, we have $\lfloor \Delta_{\overline{X}} \rfloor \leq 0$, and so $\lceil \pi^*D \rceil \geq \lfloor \Delta_{\overline{X}} + \pi^*D \rfloor$. Therefore
\begin{align*}
 K_{\overline{X}} + \overline{S} + \lceil \pi^*D \rceil &\geq \lfloor K_{\overline{X}} + \overline{S} + \Delta_{\overline{X}} + \pi^*D \rfloor \\
&= \lfloor \pi^*(K_X+S+D) \rfloor, \text{ and}\\
 K_{\overline{X}} + \lceil \pi^*D \rceil &\geq \lfloor K_{\overline{X}} + \Delta_{\overline{X}} + \pi^*D \rfloor \\
&\geq \lfloor \pi^*(K_X+D) \rfloor.
\end{align*}
In particular,
\begin{align*}
	\pi_*\mcO_{\overline{X}}(K_{\overline{X}} + \lceil \pi^*D \rceil) &= \mcO_X(K_X+D), \text{ and } \\
	\pi_*\mcO_{\overline{X}}(K_{\overline{X}} + \overline{S} + \lceil \pi^*D \rceil) &= \mcO_X(K_X+S+D).
\end{align*}

By applying $\pi_*$ to the above exact sequence, we have that
\[
0 \to \mcO_X(K_X+D) \to \mcO_X(K_X+S+D) \to \pi _*\mcO _{\overline S}(\overline G) \to  R^1\pi_* \mcO_{\overline{X}}( K_{\overline{X}} + \lceil \pi^*D \rceil)
\] 
is exact, where $\overline G \coloneq K_{\overline{S}}  + \lceil \pi^*D \rceil|_{\overline{S}}$. By localizing at one dimensional points on $X$, Lemma \ref{lemma:surface} shows that 
\[
\dim \mathrm{Supp}(R^1\pi_* \mcO_{\overline{X}}(K_{\overline{X}}+ \lceil \pi^*D \rceil )) \leq  0.
\]

Set $\Delta_S \coloneq (\pi|_{\bar S})_*((\lceil \pi^*D \rceil-\pi^*D)|_{\overline{S}})$ and $G \coloneq (\pi|_{\bar S})_*\overline G$ so that $G \sim_{\mathbb Q} K_S + \Delta_S + D|_S$. 

We claim that $\mcE \simeq \mcO _S(G)$. 
Indeed, by what we observed above, $\mcE$ is isomorphic to $ \pi _*\mcO _{\overline S}(\overline G) $ in codimension one on $S$, which in turn is isomorphic to $\mcO_S(G)$ in codimension one on $S$. Since both $\mcE$ and $\mcO_S(G)$ are reflexive sheaves of rank one, the isomorphism $\mcE \simeq \mcO _S(G)$ follows.

To conclude the proof, it is enough to verify that $\Delta_S \leq \Delta_{\mathrm{diff}}$. This can be checked by localizing at one dimensional points in $\mathrm{Sing}(X)$, and thus we conclude by Lemma \ref{lemma:surface} (see also \cite{kollaretal}*{Corti 16.6.3}).
\end{proof}

\begin{remark} Proposition \ref{proposition:3folds-ses} also holds in the following two cases:
\begin{enumerate} 
	\item if $(X,S)$ is a $n$-dimensional $\mbQ$-factorial plt pair and $D$ is a Weil divisor, defined over an algebraically closed  field of characteristic $0$, and
	\item if $(X,S)$ is a $n$-dimensional $\mbQ$-factorial plt pair and $D$ is a Weil divisor, defined over an algebraically closed  field of characteristic $p>0$ such that $(X,S)$ admits a log resolution of singularities and $(S^\nu , \Delta _{\rm diff})$ is strongly F-regular where $S^\nu \to S$ is the normalization.
\end{enumerate}

For case (1) recall that since $(X,S)$ is plt, then $S$ is normal and both $\mathcal O_X(K_X+D)$ and $\mathcal O _X(K_X+S+D)$ are Cohen Macaulay sheaves by \cite[Corollary 5.25]{KM98}. For case (2) recall that by \cite{Das15}, $S$ is normal and $(X,S)$ is purely F-regular near $S$.
\end{remark}

The following lemma, used in the above proof, was applied to localizations at non-closed points, and thus we cannot assume that $k$ is algebraically closed. However this has no impact on the proof, because the surfaces under  consideration are excellent. We refer to \cite{kollar13} and \cite{tanaka16_excellent} for the classification and basic results pertaining to excellent log canonical surfaces. Let us just note that excellent klt surfaces are $\mbQ$-factorial by \cite[Corollary 4.10]{tanaka16_excellent}, and if $(X,S)$ is a plt surface pair, then $S$ is regular (\cite[3.35]{kollar13}).
\begin{lemma} \label{lemma:surface} Let $(X,S)$ be a plt surface pair defined over an arbitrary field $k$, where $S$ is a prime divisor. Let $\Delta_{\mathrm{diff}}$ be the different, let $\pi \colon \overline{X} \to X$ be a log resolution of $(X,S)$, let $\overline{S}$ be the strict transform of $S$, let $D$ be any Weil divisor. Then 
\begin{itemize}
	\item $R^1\pi_*\mcO_{\overline{X}}(K_{\overline{X}} +  \lceil \pi^*D \rceil) = 0$, and
	\item $\Delta_S \, {\leq} \, \Delta_{\mathrm{diff}}$, where $\Delta_S \, {\coloneq}\, (  \lceil \pi^*D \rceil{-}\pi^*D)|_{\overline{S}}$, and $S$ is identified with $\overline{S}$.
\end{itemize}
\end{lemma}

 % The inversion of adjunction for surfaces holds by \cite[Theorem 5.1]{tanaka16_excellent}.

\begin{proof}
As $\lfloor \lceil \pi^*D\rceil- \pi^*D \rfloor =0$, the vanishing
\[
R^1\pi_*\mcO_{\overline{X}}(K_{\overline{X}} + \lceil \pi^*D\rceil) = 0
\]
follows by Theorem \ref{thm:relKV}. 

As for the second statement, we can restrict  ourselves to a neighbourhood of $S$ so that all singularities of $X$ lie on $S$. Therefore, \cite[3.35]{kollar13} implies that each singularity of $X$ is cyclic and $S$ is regular. Moreover, if $x \in \mathrm{Sing}(X)$ and $m_x$ is the $\mbQ$-factorial index of $X$ at $x$, then $m_xD$ is Cartier and $\Delta_{\mathrm{diff}}$ has coefficient $1 - \frac{1}{m_x}$ at $x$ (\cite{kollaretal}*{Corti 16.6.3}). Therefore, $\Delta_S \leq \Delta_{\mathrm{diff}}$,
because, by definition, $\Delta_S$ has coefficients smaller than one, and $m_x\Delta_S$ is Cartier at $x \in \mathrm{Sing}(X)$.

\end{proof}

By applying Proposition \ref{proposition:3folds-ses} for $D$ replaced by $D-(K_X+S)$, we get that
\[
0 \to \mcO_X(-S+D) \to \mcO_X(D) \to \mcO_S(G) \to 0
\]
is exact {for a $\mbQ$-divisor $0 \leq \Delta_S \leq \Delta_{\mathrm{diff}}$ and a Weil divisor $G \sim_{\mbQ} D|_S-\Delta _S$. }

\begin{corollary} \label{corollary:3folds-ses} Under the assumptions of Propostion \ref{proposition:3folds-ses}, the following sequences
\begin{align*}
0 &\to \mcO_X(K_X-(k+1)S-D) \to \mcO_X(K_X-kS-D) \to \mcO_S(G'_k) \to 0  \\
0 &\to \mcO_X(-(k+1)S+D) \to \mcO_X(-kS+D) \to \mcO_S(G_k) \to 0
\end{align*}
are exact for any $k \in \mbZ_{>0}$,  some effective $\mbQ$-divisors $\Delta_k, \Delta'_k \leq \Delta_{\mathrm{diff}}$ depending on $k$, and some integral divisors $G'_k \sim_{\mbQ}K_S+\Delta' _k-(k+1)S|_S -D|_S$ and {$G_k \sim_{\mbQ} -kS|_S + D|_S - \Delta_k$}.
\end{corollary}

\section{Proof of the main theorem}
\begin{proof}[Proof of Theorem \ref{theorem:main}]
{Fix $p_0$ as in Theorem \ref{theorem:CTW}.} Since the question is local, we can assume that $X$ is affine, and aim to show that a fixed closed point $x \in X$ is rational. We replace $X$ by a Zariski-open neighbourhood of $x$ whenever convenient. Let $f \colon Y \to X$ be a generalized plt blow-up at $x$ as in Proposition \ref{proposition:plt-blow-up}. Let $S \coloneq \Exc(f)$. Recall that $(Y,S+\Delta_Y)$ is plt, $-S$ is $f$-nef, and $-(K_Y+S+\Delta_Y)$ is $f$-ample for some effective $\mbQ$-divisor $\Delta_Y$ on $Y$. Moreover, we may assume that $Y$ is strongly F-regular. First, we show that $Rf_*\mcO_Y = \mcO_X$.

To this end, we apply Corollary \ref{corollary:3folds-ses} to get short exact sequences
\[
0 \to \mcO_Y(-(k+1)S) \to \mcO_Y(-kS) \to \mcO_S(G_k) \to 0
\]
for all $k\geq 0$ where $G_k \sim_{\mbQ}  -kS|_S-\Delta _k$ for some effective $\mbQ$-divisor $\Delta_k \leq \Delta_{\mathrm{diff}}$, where $\Delta_{\mathrm{diff}}$ is the different. Thus, to show that $R^if_*\mcO_Y(-kS)$ vanishes for $k=0$ and $i>0$, it is enough to prove that
\begin{itemize}
 	\item $R^if_*\mcO_Y(-mS) = 0$ for $i>0$ and a divisible enough $m \gg 0$, and
 	\item $H^i(S, \mcO_S(G_k))=0$ for $i>0$ and all $k\geq 0$.
 \end{itemize} 

If $X$ is $\mbQ$-factorial, then the first vanishing follows from the relative Serre vanishing, as $-S$ is $f$-ample. In general, we consider the $f$-relative semiample fibration $Y \xrightarrow{g} Y' \xrightarrow{f'} X$ associated to $-S$. It exists by Proposition \ref{proposition:non-Q-factorial-case}. Since $Y'$ is strongly F-regular along $S_{Y'} \coloneq g(S)$, we may assume that $Y'$ has rational singularities (see Proposition \ref{proposition:f-regular_are_rational}). For $m>0$ sufficiently divisible, $mS_{Y'}$ is Cartier and $mS=g^*(mS_{Y'})$. By the projection formula, we have $Rg_* \mcO_Y(-mS) =g_* \mcO_Y(-mS) =\mcO_{Y'}(-mS_{Y'})$ and therefore $R^if_*\mcO_Y(-mS) = R^if'_*\mcO_{Y'}(-mS_{Y'})$, which is zero for $m\gg 0$ by Serre vanishing since $S_{Y'} $ is $f'$-anti-ample. 

Therefore, it suffices to show that $H^i(S, \mcO_S(G_k))=0$
for $i>0$ and all $k\geq 0$. To this end, we notice that \[ (K_Y+S+\Delta_Y)|_S = K_S + \Delta_{\mathrm{diff}} + \Delta_Y|_S \] is anti-ample and $(S, \Delta_{\mathrm{diff}} + \Delta_Y|_S)$ is klt. The cohomology in question vanishes because
\[
G_k\sim _\mbQ K_S + (\Delta_{\mathrm{diff}} + \Delta_Y|_S - \Delta_k) \underbrace{-(K_S + \Delta_{\mathrm{diff}}+\Delta_Y|_S) -kS|_S}_{\text{ample}},
\]
which is zero by Theorem \ref{theorem:CTW}.\\

By an analogous argument, considering the other short exact sequence in Corollary \ref{corollary:3folds-ses}
\[
0 \to \mcO_Y(K_Y-(k+1)S) \to \mcO_Y(K_Y-kS) \to \mcO_S(G'_k) \to 0,
\]
one can show that $R^if_*\omega_Y = 0$ for $i>0$. Indeed, $Rg_* \omega_Y = \omega_{Y'}$ by Proposition \ref{prop:rational_and_CM_is_good} applied to both $Y$ and $Y'$. As $S$ is $g$-trivial, we have 
\[
R^if_*\mcO_Y(K_Y - mS) = R^if'_*\mcO_{Y'}(K_{Y'} - mS_{Y'}) = 0,
\]
 for $i>0$ and a divisible enough $m\gg 0$ by the relative Serre vanishing. Moreover, by Theorem \ref{theorem:CTW}
 \[
 H^i(S, \mcO _S (G'_k))=0\qquad \forall i>0
 \]
 for  any $k\geq 0$ and a divisor $G'_k$ such that $G'_k\sim _\mbQ K_S + \Delta' _k- (k+1)S|_S$ where $\Delta' _k $ is an effective $\mbQ$-divisor satisfying $\Delta' _k \leq \Delta_{\mathrm{diff}}$, because $-S|_S$ is nef and big (see Proposition \ref{proposition:non-Q-factorial-case}).

Proposition \ref{proposition:f-regular_are_rational} implies that $Y$ is Cohen-Macaulay, and so $X$ is Cohen-Macaulay as well by Proposition \ref{proposition:CM_kollar}. Let $\pi_Y \colon \overline{Y} \to Y$ be a log resolution. By Proposition \ref{proposition:f-regular_are_rational}, we have that $R(\pi_Y)_* \mcO_{\overline{Y}} = \mcO_Y$, and so, by the composition of derived functors, we get that $R\pi_*\mcO_{\overline{Y}} = \mcO_X$, where $\pi \coloneq f \circ \pi_Y \colon \overline{Y} \to X$. Thus $X$ has rational singularities. \\

In order to prove the last statement, we assume that $X$ is $\mbQ$-factorial, fix a Weil divisor $D$ on $X$, and set $D_Y \coloneq \lfloor f^* D \rfloor = f^*D - \epsilon S$ for some $0 \leq \epsilon < 1$. 

First, we show that $Rf_*\mcO_Y(D_Y) = \mcO_X(D)$. By Corollary \ref{corollary:3folds-ses}, we get short exact sequences
\[
0 \to \mcO_Y(-(k+1)S+D_Y) \to \mcO_Y(-kS+D_Y) \to \mcO_S(G_k) \to 0
\]
for all $k\geq 0$ where $G_k \sim_{\mbQ}  -(k+\epsilon)S|_S-\Delta _k$ for some effective $\mbQ$-divisor $\Delta_k \leq \Delta_{\mathrm{diff}}$. As in the argument above, when $i>0$ we have
\begin{align*}	
	H^i(S, \mcO_S(G_k)) &= 0, \text{ and }  \\
	R^if_*\mcO_Y(-mS+D_Y) &= 0 \text{ for divisible enough } m \gg 0,
\end{align*}
where the second identity follows from Serre's vanishing as $-S$ is $f$-ample, and the first one is a consequence of Theorem \ref{theorem:CTW} given that
\[
G_k\sim _\mbQ K_S + (\Delta_{\mathrm{diff}} + \Delta_Y|_S - \Delta_k) \underbrace{-(K_S + \Delta_{\mathrm{diff}}+\Delta_Y|_S) -(k+\epsilon)S|_S}_{\text{ample}}.
\] 
The claim now follows easily by an argument similar to the one we used in the proof that the singularities are rational.

Analogously, one can show that $R^if_*\mcO_Y(K_Y-D_Y) = 0$ for $i>0$ by considering the other short exact sequence in Corollary \ref{corollary:3folds-ses}
\[
0 \to \mcO_Y(K_Y-(k+1)S-D_Y) \to \mcO_Y(K_Y-kS-D_Y) \to \mcO_S(G'_k) \to 0.
\]
Indeed, $R^if_*\mcO_Y(K_Y - mS - D_Y) = 0$ for sufficiently divisible $m\gg 0$ by Serre vanishing. Moreover, by Theorem \ref{theorem:CTW}
 \[
 H^i(S, \mcO _S (G'_k))=0\qquad \forall i>0
 \]
 for  any $k\geq 0$ and a divisor $G'_k$ such that $G'_k\sim _\mbQ K_S + \Delta' _k- (k+1-\epsilon)S|_S$ where $\Delta' _k $ is an effective $\mbQ$-divisor satisfying $\Delta' _k \leq \Delta_{\mathrm{diff}}$.

Since $\mcO_Y(D_Y)$ is Cohen-Macaulay (see Proposition \ref{proposition:f-regular_are_rational}), Proposition \ref{proposition:CM_kollar} concludes the proof.
\end{proof}

\section{Applications and open problems} \label{section:applications}
\begin{corollary}\label{c:lift} There exists a constant $p_0 >0$ with the following property.

Let $(X,\Delta_X)$ and $(Y,\Delta_Y)$ be Kawamata log terminal threefolds defined over an algebraically closed field $k$ of characteristic $p>p_0$, and let $\pi \colon X \to Y$ be a proper birational morphism between them. If $X$ lifts to $\widetilde{X}_m$ over $W_m(k)$ for some $m \in \mbN$, then $\pi$ lifts to $\widetilde{\pi}_m \colon \widetilde{X}_m \to \widetilde{Y}_m$.
\end{corollary}
\begin{proof}
This follows from Theorem \ref{theorem:main} and Proposition \ref{proposition:pushing_lifts}.
\begin{comment} We pick $p_0$ as in Theorem \ref{theorem:main} and proceed by induction. We assume that $\pi$ lifts to ${\widetilde{\pi}}_{l-1} \colon \widetilde{X}_{l-1} \to \widetilde{Y}_{l-1}$ over $W_{l-1}(k)$ for $l \leq m$, and we would like to show that it lifts to $W_l(k)$. 

To this end, we consider the following exact sequence of sheaves of algebras on $\widetilde{Y}_{m-1}$:
\[
0 \to \pi_* \mcO_X \to ({\widetilde{\pi}}_{l-1})_*\mcO_{\widetilde{X}_l} \to ({\widetilde{\pi}}_{l-1})_*\mcO_{\widetilde{X}_{l-1}} \to R^1\pi_* \mcO_X,
\]
where $\widetilde{X}_l$ is the lift of $X$ to $W_l(k)$ induced by $\widetilde{X}_m$.

Since $X$ and $Y$ have rational singularities (Theorem \ref{theorem:main}), the Leray spectral sequence shows that $R^1\pi_* \mcO_X = 0$. Thus, we can take 
\[
\widetilde{Y}_l \coloneq \Spec\Big(({\widetilde{\pi}}_{l-1})_*\mcO_{\widetilde{X}_l}\Big).
\]
\end{comment}
\end{proof}

With Corollary \ref{c:lift} in mind, it is natural to ask the following.
\begin{question}\label{q:lift} Let $X$ be a Kawamata log terminal threefold defined over an algebraically closed field $k$ of characteristic $p\gg 0$ and $X\dasharrow X^+$ a flip. If $X$ lifts to $\widetilde X _m $ over $W_m(k)$ for some $m\in \mathbb N$, then does   $X^+$ lift to $\widetilde X ^+_m $ over $W_m(k)$?
\end{question}

\begin{remark} \label{remark:quotient} After the first version of the paper had been announced, we were informed by Professor Takehiko Yasuda that for any algebraically closed field $k$ of characteristic $p>2$ there exists a quotient klt singularity of dimension $d$ depending on $p$ which is not Cohen-Macaulay. For $p\geq 5$, his construction is the following. Let $X = V / G$, where $V$ is a $d$-dimensional $k$-vector space, $\frac{d(d-1)}{2} \geq p \geq d \geq 4$, and $G \simeq \mathbb{F}_p$ is a subgroup of $\GL_d(k)$ generated by a matrix of the form
\[
\begin{bmatrix}
    1 & 1 & 0 & 0 & \dots & 0 & 0 \\
    0 & 1 & 1 & 0 & \dots & 0 & 0 \\
    0 & 0 & 1 & 1 & \dots & 0 & 0 \\ 
    0 & 0 & 0 & 1 & \dots & 0 & 0 \\ 
    & & &  & \ddots  & \\ 
    0 & 0  & 0 & 0 & \dots & 1 & 1 \\
    0 & 0  & 0 & 0 & \dots & 0 & 1
  \end{bmatrix},
  \]
 that is ${\rm Id} + I$. Here ${\rm Id}$ is the identity matrix and $I$ is a matrix with ones just above the diagonal and zeroes elsewhere. We have $({\rm Id} + I)^p = {\rm Id}$, because $p \geq d$. Since the action of $G$ on $V$ is indecomposable and $\frac{d(d-1)}{2} \geq p$, the variety $X$ is klt by \cite[Proposition 6.6 and Proposition 6.9]{Yasuda}. Since $\dim V^G = 1 < d - 2$, it is not Cohen-Macaulay by the main result in \cite{ellingsrud-skjelbred} (see the first paragraph of the article). 

 By a similar construction one can get an example of a non-Cohen-Macaulay quotient klt singularity for $p=3$.

\end{remark}
It is also natural to wonder what the optimal value for the constant $p_0$ is.
\begin{question} Let $X$ be a three dimensional klt variety defined over an algebraically closed field of characteristic $p>2$. Does $X$ have rational singularities?
\end{question}
\section*{Acknowledgements}
We would like to thank Fabio Bernasconi, Paolo Cascini, J{\'a}nos Koll{\'a}r, S{\'a}ndor Kov{\'a}cs, Karl Schwede, Hiromu Tanaka, and Takehiko Yasuda for comments and helpful suggestions.

The first author was partially supported by NSF research grants no: DMS-1300750, DMS-1265285 and by a grant
from the Simons Foundation; Award Number: 25620. The second author was supported by the Engineering and Physical Sciences Research Council [EP/L015234/1].

\bibliographystyle{amsalpha}
\bibliography{Library}

\begin{bibdiv}
\begin{biblist}

\bib{bchm06}{article}{
      author={Birkar, C.},
      author={Cascini, P.},
      author={Hacon, C.},
      author={M\textsuperscript{c}Kernan, J.},
       title={Existence of minimal models for varieties of log general type},
        date={2010},
     journal={J. Amer. Math. Soc.},
      volume={23},
      number={2},
       pages={405\ndash 468},
}

\bib{birkar13}{article}{
      author={Birkar, Caucher},
       title={Existence of flips and minimal models for 3-folds in char {$p$}},
        date={2016},
        ISSN={0012-9593},
     journal={Ann. Sci. \'Ec. Norm. Sup\'er. (4)},
      volume={49},
      number={1},
       pages={169\ndash 212},
         url={http://dx.doi.org/10.24033/asens.2279},
}

\bib{BW17}{article}{
      author={Birkar, Caucher},
      author={Waldron, Joe},
       title={Existence of {M}ori fibre spaces for 3-folds in char {$p$}},
        date={2017},
        ISSN={0001-8708},
     journal={Adv. Math.},
      volume={313},
       pages={62\ndash 101},
         url={http://dx.doi.org/10.1016/j.aim.2017.03.032},
}

\bib{cgs14}{article}{
      author={Cascini, Paolo},
      author={Gongyo, Yoshinori},
      author={Schwede, Karl},
       title={Uniform bounds for strongly {$F$}-regular surfaces},
        date={2016},
        ISSN={0002-9947},
     journal={Trans. Amer. Math. Soc.},
      volume={368},
      number={8},
       pages={5547\ndash 5563},
         url={http://dx.doi.org/10.1090/tran/6515},
}

\bib{CP08}{article}{
      author={Cossart, Vincent},
      author={Piltant, Olivier},
       title={Resolution of singularities of threefolds in positive
  characteristic. {I}. {R}eduction to local uniformization on
  {A}rtin-{S}chreier and purely inseparable coverings},
        date={2008},
        ISSN={0021-8693},
     journal={J. Algebra},
      volume={320},
      number={3},
       pages={1051\ndash 1082},
         url={http://dx.doi.org/10.1016/j.jalgebra.2008.03.032},
}

\bib{CP09}{article}{
      author={Cossart, Vincent},
      author={Piltant, Olivier},
       title={Resolution of singularities of threefolds in positive
  characteristic. {II}},
        date={2009},
        ISSN={0021-8693},
     journal={J. Algebra},
      volume={321},
      number={7},
       pages={1836\ndash 1976},
         url={http://dx.doi.org/10.1016/j.jalgebra.2008.11.030},
}

\bib{CR15}{article}{
      author={Chatzistamatiou, Andre},
      author={R\"ulling, Kay},
       title={Vanishing of the higher direct images of the structure sheaf},
        date={2015},
        ISSN={0010-437X},
     journal={Compos. Math.},
      volume={151},
      number={11},
       pages={2131\ndash 2144},
         url={http://dx.doi.org/10.1112/S0010437X15007435},
}

\bib{CT06PLT}{article}{
      author={Cascini, Paolo},
      author={Tanaka, Hiromu},
       title={{Purely log terminal threefolds with non-normal centres in
  characteristic two}},
        date={2016},
     journal={arXiv:1607.08590v1},
}

\bib{CT07KV}{article}{
      author={Cascini, Paolo},
      author={Tanaka, Hiromu},
       title={Smooth rational surfaces violating kawamata--viehweg vanishing},
        date={2017},
        ISSN={2199-6768},
     journal={European Journal of Mathematics},
       pages={1\ndash 15},
         url={http://dx.doi.org/10.1007/s40879-016-0127-z},
}

\bib{CTW15a}{article}{
      author={Cascini, Paolo},
      author={Tanaka, Hiromu},
      author={Witaszek, Jakub},
       title={Klt del {P}ezzo surfaces which are not globally ${F}$-split},
        date={2017},
     journal={Int. Math. Res. Notices},
         url={https://doi.org/10.1093/imrn/rnw300},
}

\bib{CTW15b}{article}{
      author={Cascini, Paolo},
      author={Tanaka, Hiromu},
      author={Witaszek, Jakub},
       title={On log del {P}ezzo surfaces in large characteristic},
        date={2017},
        ISSN={0010-437X},
     journal={Compos. Math.},
      volume={153},
      number={4},
       pages={820\ndash 850},
         url={http://dx.doi.org/10.1112/S0010437X16008265},
}

\bib{ctx13}{article}{
      author={Cascini, P.},
      author={Tanaka, H.},
      author={Xu, C.},
       title={On base point freeness in positive characteristic},
        date={2015},
     journal={Ann. Sci. Ecole Norm. Sup.},
      volume={48},
      number={5},
       pages={1239\ndash 1272},
}

\bib{cutkosky09}{article}{
      author={Cutkosky, Steven~Dale},
       title={Resolution of singularities for 3-folds in positive
  characteristic},
        date={2009},
        ISSN={0002-9327},
     journal={Amer. J. Math.},
      volume={131},
      number={1},
       pages={59\ndash 127},
         url={http://dx.doi.org/10.1353/ajm.0.0036},
}

\bib{CS09}{article}{
      author={Cynk, S{\l}awomir},
      author={van Straten, Duco},
       title={Small resolutions and non-liftable {C}alabi-{Y}au threefolds},
        date={2009},
        ISSN={0025-2611},
     journal={Manuscripta Math.},
      volume={130},
      number={2},
       pages={233\ndash 249},
         url={http://dx.doi.org/10.1007/s00229-009-0293-0},
}

\bib{Das15}{article}{
      author={Das, O.},
       title={On strongly {$F$}-regular inversion of adjunction},
        date={2015},
        ISSN={0021-8693},
     journal={J. Algebra},
      volume={434},
       pages={207\ndash 226},
         url={http://dx.doi.org/10.1016/j.jalgebra.2015.03.025},
}

\bib{ellingsrud-skjelbred}{article}{
      author={Ellingsrud, Geir},
      author={Skjelbred, Tor},
       title={Profondeur d'anneaux d'invariants en caract\'eristique {$p$}},
        date={1980},
        ISSN={0010-437X},
     journal={Compos. Math.},
      volume={41},
      number={2},
       pages={233\ndash 244},
         url={http://www.numdam.org/item?id=CM_1980__41_2_233_0},
}

\bib{GNT06}{article}{
      author={Gongyo, Yoshinori},
      author={Nakamura, Yusuke},
      author={Tanaka, Hiromu},
       title={{Rational points on log Fano threefolds over a finite field}},
        date={2016},
     journal={Journal of the European Mathematical Society (to appear)},
}

\bib{hara98}{article}{
      author={Hara, N.},
       title={Classification of two-dimensional {$F$}-regular and {$F$}-pure
  singularities},
        date={1998},
        ISSN={0001-8708},
     journal={Adv. Math.},
      volume={133},
      number={1},
       pages={33\ndash 53},
         url={http://dx.doi.org/10.1006/aima.1997.1682},
}

\bib{hx13}{article}{
      author={Hacon, C.},
      author={Xu, C.},
       title={On the three dimensional minimal model program in positive
  characteristic},
        date={2015},
        ISSN={0894-0347},
     journal={J. Amer. Math. Soc.},
      volume={28},
      number={3},
       pages={711\ndash 744},
         url={http://dx.doi.org/10.1090/S0894-0347-2014-00809-2},
}

\bib{KM98}{book}{
      author={Koll{\'a}r, J.},
      author={Mori, S.},
       title={Birational {G}eometry of {A}lgebraic {V}arieties},
      series={Cambridge {T}racts in {M}athematics},
   publisher={Cambridge University Press},
        date={1998},
      volume={134},
}

\bib{kollar13}{book}{
      author={Koll{\'a}r, J.},
       title={Singularities of the minimal model program},
      series={Cambridge Tracts in Mathematics},
   publisher={Cambridge University Press, Cambridge},
        date={2013},
      volume={200},
        ISBN={978-1-107-03534-8},
         url={http://dx.doi.org/10.1017/CBO9781139547895},
}

\bib{kovacs17}{article}{
      author={Kov\'acs, S.},
       title={Non-{C}ohen-{M}acaulay canonical singularities},
        date={2017},
     journal={arXiv:1703.02080v2},
}

\bib{kovacs17b}{article}{
      author={Kov\'acs, S.},
       title={Super-rational singularities},
        date={2017},
     journal={arXiv:1703.02269v3},
}

\bib{kollaretal}{book}{
      author={Koll{\'a}r, J.},
      author={others},
       title={Flips and abundance for algebraic threefolds},
   publisher={Soci{\'e}t{\'e} Math{\'e}matique de France},
     address={Paris},
        date={1992},
}

\bib{lip78}{article}{
      author={Lipman, Joseph},
       title={Desingularization of two-dimensional schemes},
        date={1978},
     journal={Ann. Math. (2)},
      volume={107},
      number={1},
       pages={151\ndash 207},
}

\bib{LS14}{article}{
      author={Liedtke, Christian},
      author={Satriano, Matthew},
       title={On the birational nature of lifting},
        date={2014},
        ISSN={0001-8708},
     journal={Adv. Math.},
      volume={254},
       pages={118\ndash 137},
         url={http://dx.doi.org/10.1016/j.aim.2013.10.030},
}

\bib{patakfalvi-schwede14}{article}{
      author={Patakfalvi, Zsolt},
      author={Schwede, Karl},
       title={Depth of {$F$}-singularities and base change of relative
  canonical sheaves},
        date={2014},
        ISSN={1474-7480},
     journal={J. Inst. Math. Jussieu},
      volume={13},
      number={1},
       pages={43\ndash 63},
         url={http://dx.doi.org/10.1017/S1474748013000066},
}

\bib{schwedesmith10}{article}{
      author={Schwede, K.},
      author={Smith, K.~E.},
       title={Globally {$F$}-regular and log {F}ano varieties},
        date={2010},
        ISSN={0001-8708},
     journal={Adv. Math.},
      volume={224},
      number={3},
       pages={863\ndash 894},
         url={http://dx.doi.org/10.1016/j.aim.2009.12.020},
}

\bib{schwedetucker12}{incollection}{
      author={Schwede, K.},
      author={Tucker, K.},
       title={A survey of test ideals},
        date={2012},
   booktitle={Progress in commutative algebra 2},
   publisher={Walter de Gruyter, Berlin},
       pages={39\ndash 99},
}

\bib{tanaka12}{article}{
      author={Tanaka, H.},
       title={Minimal models and abundance for positive characteristic log
  surfaces},
        date={2014},
        ISSN={0027-7630},
     journal={Nagoya Math. J.},
      volume={216},
       pages={1\ndash 70},
         url={http://dx.doi.org/10.1215/00277630-2801646},
}

\bib{tanaka16_excellent}{article}{
      author={Tanaka, Hiromu},
       title={{Minimal Model Programme for excellent surfaces}},
        date={2016},
     journal={arXiv:1608.07676v2},
}

\bib{Yasuda}{article}{
      author={Yasuda, Takehiko},
       title={The {$p$}-cyclic {M}c{K}ay correspondence via motivic
  integration},
        date={2014},
        ISSN={0010-437X},
     journal={Compos. Math.},
      volume={150},
      number={7},
       pages={1125\ndash 1168},
         url={http://dx.doi.org/10.1112/S0010437X13007781},
}

\end{biblist}
\end{bibdiv}

\end{document}